\documentclass[]{article}
\usepackage{indentfirst}
\usepackage{amsmath}
\usepackage{authblk}
\usepackage{geometry}
\usepackage{accents}
\usepackage{statex}
\usepackage[francais,english]{babel}
\usepackage{url}

\usepackage[utf8]{inputenc} 
\usepackage[nottoc]{tocbibind}

\newtheorem{theorem}{Theorem}
\newtheorem{proposition}{Proposition}
\newtheorem{lemma}{Lemma}

\newenvironment{proof}[1][Proof]{\begin{trivlist}
		\item[\hskip \labelsep {\bfseries #1}]}{\end{trivlist}}
\newenvironment{definition}[1][Definition]{\begin{trivlist}
		\item[\hskip \labelsep {\bfseries #1}]}{\end{trivlist}}
\newenvironment{example}[1][Example]{\begin{trivlist}
		\item[\hskip \labelsep {\bfseries #1}]}{\end{trivlist}}

\title{Automatic sequences defined by Theta functions and some infinite products}
\author{Shuo LI}
\date {}

\begin{document}
	
\maketitle

\section{Introduction}
Let $p(x) \in C(x)$ be a rational function satisfying the condition $p(0)=1$ and $q$ an integer larger than $1$, in this article we will consider the expansion in power series of the infinite product 

$$f(x)=\prod_{s=0}^{\infty}p(x^{q^{s}})=\sum_{i=0}^{\infty}c_ix^i,$$
 and study when the sequence $(c_i)_{i \in \mathbf{N}}$ is $q$-automatic. This topic has been studied by many authors, such as \cite{dumas}, \cite{duke2015} and \cite{Checcoli2018} , using analytical approach, here we want to review this topic by a basic algebraic approach.

The main result is that for given integers $q \geq 2$ and $d \geq 0$, there exist finitely many polynomials of degree $d$ defined over the field of rational numbers $\mathbf{Q}$, such that $f(x)=\prod_{s=0}^{\infty}p(x^{q^{s}})=\sum_{i=0}^{\infty}c_ix^i$ is a $q$-automatic power series. 

\section{Definitions and generality}

\begin{definition} 
Let $(a_n)_{n \in \mathbf{N}}$ be a sequence, we say it is $q$-automatic if the set 
$$Ker((a_i)_{i \in \mathbf{N}})=\left\{(a_{q^ln+b})_{n \in \mathbf{N}}|l \in \mathbf{N}, 0\leq b < q^l\right\}$$
is finite. This set will be called the $q$-kernel of $(a_n)_{n \in \mathbf{N}}$.\\

For every couple of  integers $(l,b)$ satisfying $l \in \mathbf{N}, 0\leq b < q^l$, let us define a relation $R_{l,b}$ over the sequence space: we say $R_{l,b}((a_n)_{n \in \mathbf{N}},(b_n)_{n \in \mathbf{N}})$ if and only if
$$\forall n \in \mathbf{Z}, b_n=a_{q^ln+b}.$$
\end{definition}

\begin{definition} 
Let $\sum_{i=0}^{\infty}a_ix^i$ be a  power series, we say it is $q$-automatic if the sequence of coefficients $(a_n)_{n \in \mathbf{N}}$ is $q$-automatic. 

Similarly we define operators $O_{l,b}$ over the space of power series:

$$O_{l,b}(\sum_{n=0}^{\infty}a_nx^n)=\sum_{n=0}^{\infty}a_{q^ln+b}x^n.$$

\end{definition}

Now let us consider a detailed version of a well-known theorem, see, for example, \cite{Allouche}.

\begin{proposition}
let $f \in F((x))$ be a $k$-automatic power series, then there exist polynomials $a_0(x),a_1(x),...,a_m(x) \in F[x]$ with $a_0(x)a_m(x)\neq 0$ such that $$\sum_{i=0}^{m}a_i(x)f(x^{k^i})=0.$$
Furthermore, the coefficients of $a_0(x),a_1(x),...,a_x(t)$ depend only on $R_{l,b}$ relations over the $q$-kernel  of the sequence of the coefficients of $f$. 
\end{proposition}

\begin{proof}
Let $B$ denote the $k$-kernel of the sequence of coefficients of $f$, and $N$ denote the cardinal of $B$. We can then associate each element in $B$ with a power series by

$$(a_n)_{n \in \mathbf{N}} \to \sum_{n=0}^{\infty}a_nx^n.$$

Let $B'$ denote the image of $B$ by the previous map. For each power series in $B'$, we have 

$$\sum_{i=0}^{\infty}a_ix^i=\sum_{i=0}^{k-1}x^i(\sum_{j=0}^{\infty}a_{kj+i}x^{kj}).$$

Remarking that if the sequence $(a_n)_{n \in \mathbf{N}}$ is in $B$, then $(a_{kn+j})_{n \in \mathbf{N}}$ is also in $B$, for $j=0, 1, ..., k-1$. If we write
$$\sum_{i=0}^{\infty}a_ix^i=\sum_{(b_n)_{n \in \mathbf{N}} \in B'}c_b\sum_{i=0}^{\infty}b_ix^{ki},$$
Then $$c_b=\begin{cases}
x^i, \; \text{if} \; R_{1,i}((a_n)_{n \in \mathbf{N}},(b_n)_{n \in \mathbf{N}})\\
0, \; \text{otherwise.}
\end{cases}$$

Particularly, we can do the same thing for $f(x),f(x^k),...,f(x^{k^{N}})$:

$$\begin{cases}
f(x)=\sum_{(b_n)_{n \in \mathbf{N}} \in B}c_b^1\sum_{i=0}^{\infty}b_ix^{k^{N+1}i},\\
f(x^k)=\sum_{(b_n)_{n \in \mathbf{N}} \in B}c_b^2\sum_{i=0}^{\infty}b_ix^{k^{N+1}i},\\
...\\
f(x^{k^N})=\sum_{(b_n)_{n \in \mathbf{N}} \in B}c_b^N\sum_{i=0}^{\infty}b_ix^{k^{N+1}i};
\end{cases}$$
with $c_b^j$ defined only by $R_{l,b}$ relations. But as the cardinal of  $B'$ is $N$, the linear forms at the right-hand side of above equalities are linearly dependent. As a result, if we neglect the linear dependence between elements in $B'$, we can have a linear dependence between $f(x),f(x^k),...,f(x^{k^{N}})$ such that the coefficients depend only on  $c_b^j$. So these coefficients depend only on $R_{l,b}$ relations.
\end{proof}

Here we make this proposition precise by some examples:\\

\begin{example}
Let us consider a periodic sequence
$$a,b,a,b,a,b,a,b...$$
which is $2$-automatic.

Now let us write down the associated power sequence $F(x)=a+bx+ax^2+bx^3+...$ and two other sequences $A(x)=a+ax+ax^2+ax^3...$, $B(x)=b+bx+bx^2+bx^3...$ with constant coefficients.

So $$F(x)=A(x^2)+xB(x^2)$$
$$A(x)=(1+x)A(x^2)$$
$$B(x)=(1+x)B(x^2)$$

so we have the following dependence:

$$F(x)=(1+x^2)(1+x^4)A(x^8)+x(1+x^2)(1+x^4)B(x^8)$$
$$F(x^2)=(1+x^4)A(x^8)+x^2(1+x^4)B(x^8)$$
$$F(x^4)=A(x^8)+x^4B(x^8)$$

$F(x)$ satisfies the functional equation $$(x^8-x^6+x^4-x^2)((1+x^2)F(x^2)-F(x))=(x^4-x^3+x^2-x)(1+x^4)((1+x^4)F(x^4)-F(x^2))$$

This functional equation does not depend on the values of $a$ and $b$.
\end{example}

\begin{example}
Let us consider the Thue-Morse sequence
$$a,b,b,a,b,a, a, b, b, a, a, b, a, b, b, a...$$
which is $2$-automatic.

Now let us write down the associated power sequence $F(x)=a+bx+bx^2+ax^3+...$ and another sequence $G(x)=b+ax+ax^2+bx^3...$,  by changing $a$ to $b$ and $b$ to $a$:

So $$F(x)=F(x^2)+xG(x^2)$$
and $$G(x)=G(x^2)+xF(x^2)$$

so we have the following dependence:

$$G(x^2)=G(x^4)+x^2F(x^4)$$
$$x^2G(x^4)=F(x^2)-F(x^4)$$
$$x^2G(x^2)=x^2G(x^4)+x^4F(x^4)$$

$F(x)$ satisfies the functional equation $$(x^4-1)F(x^4)+(1+x)F(x^2)-xF(x)=0$$

This functional equation does not depend on the values of $a$ and $b$.
\end{example}

\begin{proposition}
For a given functional equation $F: \sum_{s=0}^{m}a_s(t)f(t^{k^s})=0$,  there exist finitely many polynomials $p_1, p_2,...,p_r$  with $p_i(0)=1, \forall i \in [0,r]$, such that the associated theta functions $G_r(x)= \prod_{s=0}^{\infty}p_r(x^{q^{s}})$ satisfying equation $F$.
\end{proposition}

\begin{proof}
 If $p(x)$ is a such polynomial satisfying $p(0)=1$. Let us denote by $G(x)$ the associated power series. By hypothesis, it satisfies the functional equation $F$:
 $$\sum_{s=0}^ma_s(x)G(x^{q^s})=0.$$
 On the other hand, the power series $G$ satisfies another functional equation:
 $$G(x)=p(x)G(x^q).$$
 Plugging the second equation into the first one, we get
 $$\sum_{s=0}^{m}a_s(x)\prod_{r=s}^{m}p(x^{q^{m-r}})=0.$$
 An observation is that all terms in the sum contain a factor $p(x^{q^{m-1}})$ except the last one. So we have
 $$p(x^{q^{m-1}})|a_m(x)$$ 
 with $p(0)=1$, so there are finitely many choices for $p(x)$.
\end{proof}

\begin{proposition}
For a fixed number $k$, there are finitely many polynomials $p_1, p_2,...,p_r$ such that the theta functions $G_j(x)= \prod_{s=0}^{\infty}p_j(x^{q^{s}})$ are $q$-automatic and the sizes of their $q$-kernels are bounded by $k$.
\end{proposition}

\begin{proof}
 Fixing the size of the $q$-kernel, we fix the number of possibilities of $R_{l,b}$ relations, so the possible functional equations, and we conclude by Proposition 4.2.
\end{proof}

\section{Infinite product of polynomials}

Let $p=\sum_{i=1}^na_ix^i$ be a polynomial with coefficients in  $\mathbf{C}$ and $q$ be an integer larger than $1$. It is known that the coefficients of the power series 
$$f(x)= \prod_{s=0}^{\infty}p(x^{q^{s}})$$
form a $q$-regular sequence \cite{dumas}, here we want to study when this sequence is $q$-automatic. 

Firstly, let us suppose that the degree of $p$, noted $deg(p)$, satisfies $q^{k-1}<deg(p)\leq q^k$ for some $k \in \mathbf {N}$ and write

$$f(x)=\prod_{s=0}^{\infty}p(x^{q^{s}})=\sum_{i=1}^{\infty}c_ix^i.$$
Then the coefficients $c_i$ satisfy a recurrence relation:

 \begin{equation} 
\begin{aligned} 
c_{nq+r}=\sum_{\substack{0 \leq j \leq q^k)\\ j \equiv r \pmod q} } a_{j}c_{n+\frac{r-j}{q}}
\end{aligned}
\end{equation}
for all $r$ such that $0 \leq r \leq q-1$ and $c_n=0$ for all negative indices.

\begin{lemma}
The sequences $(c_{qn+i-j})_{n \in \mathbf{N}}$, for all $i$ and $j$ such that $0\leq i\leq q-1$ and $0\leq j\leq 2q^{k}$, can be represented as linear combinations of sequences $\left\{(c_{n-i})_{n \in \mathbf{N}}|0\leq i< 2q^{k}\right\}$.
\end{lemma}

\begin{proof}
Because of the previous equality,  we have 
$$c_{nq+i-j}=\sum_{\substack{0 \leq s \leq q^k\\ s \equiv i-j \pmod q} } a_{s}c_{n+\frac{i-j-s}{q}}$$
for all $n,i,j$ defined as above. Now let us check that all sequences appearing on the right-hand side of these equalities are in the set defined in the statement. It is enough to calculate the shifting indices and we have the bounds as follows,

$$-2q^k<-3q^{k-1}\leq \frac{i-j-s}{q}\leq 0$$
which proves the statement.
\end{proof}

\begin{example}
Let us consider the case where $p(x)=1+x+x^2+x^3+x^4$ and $q=2$, the sequence of coefficients of the power series $F(x)=\prod_{s=0}^{\infty}p(x^{q^{s}})$ is denoted by $(c_n)_{n \in \mathbf{N}}$, 
so we have
$$p(x)=(1+x+x^2+x^3+x^4)F(x^2)$$
from which we can deduce
$$c_{2n}=c_n+c_{n-1}+c_{n-2},$$
$$c_{2n+1}=c_n+c_{n-1}.$$
 Using the above lemma, we get
 
 $$\begin{pmatrix} 
c_{2n} \\
c_{2n-1}\\ 
c_{2n-2}\\
c_{2n-3} \\
c_{2n-4}\\ 
c_{2n-5}\\
c_{2n-6}\\ 
c_{2n-7}\\
c_{2n-8} \\
\end{pmatrix}=
\begin{pmatrix} 
1&1&1&0&0&0&0&0&0\\
0&1&1&0&0&0&0&0&0\\
0&1&1&1&0&0&0&0&0\\
0&0&1&1&0&0&0&0&0\\
0&0&1&1&1&0&0&0&0\\
0&0&0&1&1&0&0&0&0\\
0&0&0&1&1&1&0&0&0\\
0&0&0&0&1&1&0&0&0\\
0&0&0&0&1&1&1&0&0\\
\end{pmatrix}
\begin{pmatrix} 
c_{n} \\
c_{n-1}\\ 
c_{n-2}\\
c_{n-3} \\
c_{n-4}\\ 
c_{n-5}\\
c_{n-6}\\ 
c_{n-7}\\
c_{n-8} \\
\end{pmatrix}$$

and 

 $$\begin{pmatrix} 
c_{2n+1} \\
c_{2n}\\ 
c_{2n-1}\\
c_{2n-2} \\
c_{2n-3}\\ 
c_{2n-4}\\
c_{2n-5}\\ 
c_{2n-6}\\
c_{2n-7} \\
\end{pmatrix}=
\begin{pmatrix} 
1&1&0&0&0&0&0&0&0\\
1&1&1&0&0&0&0&0&0\\
0&1&1&0&0&0&0&0&0\\
0&1&1&1&0&0&0&0&0\\
0&0&1&1&0&0&0&0&0\\
0&0&1&1&1&0&0&0&0\\
0&0&0&1&1&0&0&0&0\\
0&0&0&1&1&1&0&0&0\\
0&0&0&0&1&1&0&0&0\\
\end{pmatrix}
\begin{pmatrix} 
c_{n} \\
c_{n-1}\\ 
c_{n-2}\\
c_{n-3} \\
c_{n-4}\\ 
c_{n-5}\\
c_{n-6}\\ 
c_{n-7}\\
c_{n-8} \\
\end{pmatrix}$$
 
\end{example}
Because of the previous fact, we can introduce some transition matrices: for all integers $r$ such that $0 \leq r \leq q-1$ let us define $\Gamma_r$ as a square matrix of size $2q^k+1$ satisfying
$$\Gamma_r
\begin{pmatrix}
    c_n      \\
    c _{n-1} \\
    ...\\
    c_{n-2q^{k}}     
\end{pmatrix}
= 
\begin{pmatrix}
  c_{qn+r}      \\
    c _{qn+r-1} \\
    ...\\
    c_{qn+r-2q^{k}}     
\end{pmatrix} $$
for all  $n \in \mathbf{N}$.

Let us denote by $G$ the semi-group generated by all $\Gamma_r$ and multiplication. 

\begin{proposition}
$a \in \left\{c_n| n \in \mathbf{N}\right\}$ if and only if there exists a matrix $g \in G$ such that $a$ is the first element in the first row of the matrix $g$, in other words, $a=g(1,1)$. Furthermore, $(c_n)_{n \in \mathbf{N}}$ is automatic if and only if $G$ is a finite semi-group.

\end{proposition}

\begin{proof}
The first part of this proposition is trivial,  for any $r \in \mathbf{N}$, let us consider its $q$-ary expansion $r=\overline{s_{k_1}s_{k_1-1}...s_0}$. Using Lemma 4.1, we have

$$\begin{pmatrix} 
c_{r} \\
c_{r-1}\\ 
...\\
c_{r-2q^{k}}  \\
\end{pmatrix}=
\Gamma_{s_{k_1}}\Gamma_{s_{k_1-1}}..\Gamma_{s_0}\begin{pmatrix} 
1 \\
0\\ 
...\\
0\\
\end{pmatrix},$$
which proves the first part of the statement.\\

For the second part, let us define maps $\gamma_r$ for all integers $r$ by $\gamma_r(n)=q(q(...q(q(n)+s_0)...)+s_{k_1-1})+s_{k_1}$ for all $n \in \mathbf{N}$ if $r=\overline{s_{k_1}s_{k_1-1}...s_0}$. Then there is an equality for all $r$:

$$\begin{pmatrix} 
c_{\gamma_r(0)} & c_{\gamma_r(1)} & ...&c_{\gamma_r(2q^k)}\\
c_{\gamma_r(0)-1}&c_{\gamma_r(1)-1}&...&c_{\gamma_r(2q^k)-2q^k}\\ 
...\\
c_{\gamma_r(0)-2q^k}&c_{\gamma_r(1)-2q^k}&...&c_{\gamma_r(2q^k)-2q^k} \\
\end{pmatrix}=
\Gamma_{s_{k_1}}\Gamma_{s_{k_1-1}}..\Gamma_{s_0}
\begin{pmatrix}
a_{0} & a_{1} & ...&a_{2q^k}\\
0&a_{0}&...&a_{2q^k-1}\\ 
...\\
0&0&...&a_0\\
\end{pmatrix}.
$$

But the the last matrix in the above equality is constant and invertible, so each element of a matrix $g \in G$ is a finite linear composition of elements in the sequence $(c_n)_{n \in \mathbf{N}}$, so the finiteness of elements in $(c_n)_{n \in \mathbf{N}}$ is equivalent to the finiteness of elements in $G$. And using the fact that $(c_n)_{n \in \mathbf{N}}$ is an automatic sequence, we conclude the statement.
\end{proof}

\begin{proposition}
For given integers $q \geq 2$ and $d \geq 0$, there exist finitely many polynomials of degree $d$ defined over the field of rational numbers $\mathbf{Q}$, such that $\prod_{s=0}^{\infty}p(x^{q^{s}})=\sum_{i=1}^{\infty}c_ix^i$ is a $q$-automatic power series.
\end{proposition}

\begin{proof}
Suppose that the sequence $(c_n)_{n \in \mathbf{N}}$ generated by $\prod_{s=0}^{\infty}p(x^{q^{s}})=\sum_{i=1}^{\infty}c_ix^i$ is automatic. Let us consider a sequence of matrices $(\Gamma_n)_{n \in \mathbf{N}}$, such that $\Gamma_i$ are defined as above for $i=0,1,.., q-1$ and $\Gamma_{qi+j}=\Gamma_i\Gamma_j$ for all $i \geq 1$ and $j=0,1,...,q-1$.

It is easy to see that this matrix sequence is automatic because $G$ is finite. And also the automata of this matrix sequence is the same as the one of $(c_n)_{n \in \mathbf{N}}$, because $c_n$ is exactly the element at the position $(1,1)$ of the matrix $\Gamma_n$. To conclude the statement, we have to prove two things: firstly the number of automata generating the sequences $(\Gamma_n)_{n \in \mathbf{N}}$ is finite, secondly, the output functions for each automaton are also finite. 

For the first point, it is enough to show that $|G|$ is bounded by a function depending only on $d$ and $q$, which is proved by Theorem 1.3 of \cite{ofsm}. It says that given naturals $n$ and $k$, there exist, up to semi-group isomorphism, only
a finite number of finite sub-semi-groups of $M_{n}(F)$ generated by at most $k$ elements. 

For the second point, it is a consequence of Proposition 4.3.
\end{proof}

\begin{proposition}
Let $f$ be a polynomial satisfying the hypothesis in Proposition 4.5, then all its coefficients belong to $\mathbf{Z}$.
\end{proposition}

\begin{proof}
Let us denote by $d$ the degree of $f$ and write down all coefficients of $f$ in the form $a_i=\frac{p_i}{q_i}$ such that $(p_i, q_i)=1$, and similarly for all coefficient of $F$, let us write down $c_i=\frac{r_i}{t_i}$ with $(r_i, t_i)=1$. If there are some coefficients of $f$ which are rational numbers but not integers, then there exist a prime $p$ and two integers $d_1$ and $d_2$ satisfying :
$$d_1=\max \left\{t| t \in \mathbf{N}, \exists q_i, p^t| q_i\right\}$$
and
$$d_2=\max \left\{t| t \in \mathbf{N},  \exists t_i, p^t| t_i\right\}$$
with $d_1>0, d_2>0$. In fact, because of the hypothesis, there exists $a_i=\frac{p_i}{q_i}$ with $q_i \neq 1$. So there exists a prime $p$ such that $p|q_i$, thus $d_1\neq 0$. Let us suppose $a_j=\frac{p_j}{q_j}$ with the smallest index such that $p^{d_1}|q_j$. Now let us check $$c_j=a_j+\sum_{qk+s=j, k>0}a_kc_s.$$ If $c_j=\frac{r_j}{t_j}$ with $p|t_j$ then $d_2 \geq 1$; otherwise, there are some $a_k,c_j$ such that $p^{d_1}|q_kt_j$, but with the assumption of smallest index, $p^{d_1}\nmid q_k$, so $p|t_j$ thus $d_2 \geq 1$.    

Let $l_1$ be the smallest index such that $p^{d_1}|q_{l_1}$ and similarly let $l_2$ be the smallest index such that $p^{d_2}|s_{l_2}$. Now let us consider the coefficient $c_{l_2q+l_1}$, which can be calculated as
$$c_{l_2q+l_1}=\sum_{0 \leq i \leq d, qj+i=l_2q+l_1} a_ic_j.$$
Let us consider the sum at the right-hand side, for any couple of $(a_i,c_j)$, if $i< t_1$, then $p^{d_1} \nmid q_i$, the maximality of $d_2$ leads to $p^{d_1+d_2} \nmid q_it_j$; similarly, if $i> t_1$, then $j <t_2$ thus $p^{d_2} \nmid t_j$, so that $p^{d_1+d_2} \nmid q_it_j$; but if $i=t_1$, then $j=t_2$, so $p^{d_1} | q_i$ and $p^{d_2} | t_i$. As a result, $p^{d_1+d_2} | c_{t_2q+t_1}$, contradicts the maximality of $d_2$.
\end{proof}

\section{Rational functions generated by infinite products}
Here we consider the following question: for a given polynomial $p$ and an integer $q$, when does $F(x)= \prod_{s=0}^{\infty}p(x^{q^{s}})$ equal a rational function. This question has already been studies in \cite{duke2015} when restricting the polynomial to the cyclotomic case, this section can be considered as a generalization of the previous work. 

\begin{proposition}
Let $p$ be a polynomial taking coefficients over $\mathbf{C}$ and $q$ be an integer larger than $1$, then there is an equivalence between:

(1) $\prod_{s=0}^{\infty}p(x^{q^{s}})$ is a rational function.

(2) there exists a polynomial $Q(x)$ such that $p(x)=\frac{Q(x^q)}{Q(x)}$ and all roots of $Q(x)$ are roots of unity, if $\delta$ is a root of $Q(x)$ then $\delta^{q^t}$ is a root of $Q$ for all $t \in \mathbf{N}$.

\end{proposition}

\begin{proof}
(2) implies (1) is straightforward, let us check (1) implies (2).

Let $F(x)= \prod_{s=0}^{\infty}p(x^{q^{s}})$ be a rational function, say $F(x)=\frac{P(x)}{Q(x)}$, where $P(x)$ and $Q(x)$ are coprime, using the functional equation $F(x)=p(x)F(x^q)$, we get $$\frac{P(x)Q(x^q)}{P(x^q)Q(x)}=p(x).$$  As $deg(p(x))>0$, so that $deg(Q(x))>deg(P(x))$, and $P(x^q)|P(x)Q(x^q)$ if $deg(P(x))>0$, then $P(x^q)$ and $Q(x^q)$ should have at least one common root, which contradicts that $P(x)$ and $Q(x)$ are coprime, so we have $$F(x)=\frac{1}{Q(x)}$$ and $$p(x)=\frac{Q(x^q)}{Q(x)}$$

Now let us study the roots of $Q(x)$, let us suppose $0\leq |r_1| \leq |r_2| \leq... \leq |r_m|$ where $r_i$ are the roots of $Q(x)$ and $|r_i|$ is the modulus of $r_i$. Firstly $|r_m|$ can not be too large, if $|r_m|>1$ then each root of $Q(x^q)$ should have a modulus strictly smaller than $|r_m|$, on the other hand $Q(x)|Q(x^q)$, which is impossible. For the same reason, $|r_1|$ can not be a real number between $0$ and $1$. So $|r_i|$ are either $0$ or $1$, but if $x|Q(x)$, the infinite product of $p(x)$ will not converge, so $|r_i|=1$ for all roots of $Q(x)$. Using once more $Q(x)|Q(x^q)$, if $\delta$ is a root of $Q(x)$ then it is a root of $Q(x^q)$ which implies $\delta^q$ is a root of $Q(x)$, we can do it recursively and we obtain $\delta^{q^t}$ is a root of $Q$ for all $t \in \mathbf{N}$, as a corollary, $\delta$ can only be a root of unity. So we prove (2) using (1).

\end{proof}

\section{Infinite product of inverse of polynomials}

In this section, we consider the power sequence defined as follows:
$$F(x)= \prod_{s=0}^{\infty}\frac{1}{p(x^{q^{s}})}=\sum_{i=0}^{\infty}c_ix^i,$$
where $q$ is an integer larger than $1$ and $p=\sum_{i=0}^{n}b_ix^i$ is a polynomial such that $p(0)=1$ defined as before.\\
Such a sequence satisfies the functional equation 
$$F(x)=\frac{1}{p(x)}F(x^q).$$
If we write $\frac{1}{p(x)}=\sum_{i=0}^{\infty} a_ix^i$, then 
$$c_{qn+i}=\sum_{j=0}^{n}a_{qj+i}c_{n-j},$$
for all $n \in\mathbf{N}$ and $i$ such that $0 \leq i \leq q-1$.
\begin{proposition}
If the coefficients of the power series $F(x)= \prod_{s=0}^{\infty}\frac{1}{p(x^{q^{s}})}=\sum_{i=0}^{\infty}c_ix^i$ take finitely many values in $\mathbf{C}$, then the roots of $f$ are all of modulus $1$. 

\end{proposition}

\begin{proof}

Firstly, let us prove that the moduli of all roots of $p$ are not smaller than $1$. Otherwise, let us chose one of those which have smallest modulus, say $\alpha$, because of the above definition, we can conclude that 

$$p(\alpha^k) \neq 0$$
for all $k$ larger than $1$.

Let us consider the equality,

$$\prod_{s=0}^{\infty}\frac{1}{p(x^{q^{s}})}=\sum_{i=0}^{\infty}c_ix^i,$$
the right-hand side converges when $x$ tends to $\alpha$ while the left-hand side diverges, in fact
$\prod_{s=1}^{\infty}\frac{1}{p(\alpha^{q^{s}})}$ converges to a non-zero value because

$$\log(\prod_{s=1}^{\infty}\frac{1}{p(\alpha^{q^{s}})})=-\sum_{s=1}^{\infty}\log(p(\alpha^{q^{s}}))$$
which converges, however, $\frac{1}{p(x^{q^{s}})}$ has a pole at $x= \alpha$.\\

Secondly, let us prove that the moduli of all roots of $p$ are not larger than $1$. Otherwise, let us chose one of them, say $\beta$, and an integer $t$ such that $|\beta|^{q^t} >|a|/|b|+1$, where $|a|$ is the largest modulus of the sequence $(c_i)_{i \in \mathbf{N}}$ and $|b|$ is the smallest non-zero modulus of this sequence. Now consider the following series

$$\frac{1}{1-\beta}\prod_{s=t}^{\infty}\frac{1}{p(x^{q^{s}})}=\sum_{i=0}^{\infty}d_ix^i.$$
It is easy to see that $\left\{d_i|i \in \mathbf{N}\right\}$ is finite, because such a series can be obtained by multiplying a polynomial to $F(x)$, but on the other hand, we have the inequality,

$$|d_{q^ti}|=|\sum_{j=0}^{i}\beta^{q^tj}c_{q^t(i-j)}| \geq -|a|\sum_{j=0}^{i-1}|\beta^{q^tj}|+|b||\beta^{q^ti}|>0$$
which diverges. This contradicts the fact that $\left\{d_i|i \in \mathbf{N}\right\}$ is finite. In conclusion, the roots of $f$ are all of modulus $1$.

\end{proof}

\begin{proposition}
If the power series $F(x)= \prod_{s=0}^{\infty}\frac{1}{p(x^{q^{s}})}=\sum_{i=0}^{\infty}c_ix^i$ is a $q$-regular sequence, then the roots of $p$ are all roots of unity, furthermore, the order of each root is multiple of $q$. 

\end{proposition}

\begin{proof}
 If $F(x)= \prod_{s=0}^{\infty}\frac{1}{p(x^{q^{s}})}=\sum_{i=0}^{\infty}c_ix^i$ is a $q$-regular sequence, then $F'(x) =\sum_{i=1}^{\infty}c_iix^{i-1}$ is also $q$-regular. On the other hand, we know $\frac{1}{F(x)}=\prod_{s=0}^{\infty}p(x^{q^{s}})$ is $q$-regular, so
 
 $$\frac{F'(x)}{F(x)}=(\log F(x))'$$
 is $q$-regular. In the same way we have $(\log F(x^q))')$ is $q$-regular so that 
 
 $$(\log F(x))'-(\log F(x^q))'=\frac{p'(x)}{p(x)}$$
 is $q$-regular, then we conclude by Theorem 3.3 \cite{Allouche} that all roots are roots of unity.\\
 
 To prove the second part, we use a method introduced in \cite{BECKER}. We firstly define some notation. Let us denote by $A_{t,i}$ the operator of power series:
 
 $$A_{t,i}(\sum_{j=0}^{\infty}a_jx^j)=\sum_{j=0}^{\infty}a_{q^tj+i}x^{q^tj+i}$$
 for all $i$ such that $0 \leq i \leq q^t-1$.\\
 
 If there exists a root of $p$ which's order is not a multiple of $q$, say $\alpha$, then for all formal power series $f$, let us define $ord(f(x))$ to be the order of pole of $f$ at point $\alpha$. It is easy to check that there exists a $t \in \mathbf{N}$ such that for all $f \in F[[x]]$, $ord(f(x))=ord(f(x^{q^t}))$ so there are some $i$ such that $ord(f(x)) \leq ord(A_{t,i}(f(x)))$.\\
 
 Now let us define a sequence of power series $(s_i)_{i \in \mathbf{N}}$ and a sequence of integer $(I_i)_{i \in \mathbf{N}}$ such that $s_0=1$, $0 \leq I_i \leq q^t-1, \forall i$ and $ord(A_{I_i}(\frac{s_i}{p(x)})) \geq ord(\frac{s_i}{p(x)})$ and we define $s_{i+1}=A_{I_i}(\frac{s_i}{p(x)})$, so we can easily check
 
 $$A_{I_i}(s_iF(x))=A_{I_i}(\frac{s_i}{p(x)})F(x)=s_{i+1}F(x),$$
 and by induction
 
 $$A_{I_i}A_{I_{i-1}}...A_{I_0}(F(x))=s_{i+1}F(x).$$
 However, $$ord(s_i) < ord(s_{i+1}),$$ the sequence $s_i$ are linearly independent, so $F(x)$ can not be a regular sequence. 
\end{proof}

\begin{theorem}
If the power series $F(x)= \prod_{s=0}^{\infty}\frac{1}{p(x^{q^{s}})}=\sum_{i=0}^{\infty}c_ix^i$ is a $q$-regular sequence, then there exists a polynomial $Q(x)$ such that $p(x)| \frac{Q(x^q)}{Q(x)}$, so $F(x)$ can be written as 

$$F(x)=Q(x)\prod_{i=1}^{\infty}R(x^q),$$

where $R(x)=\frac{Q(x^q)}{Q(x)F(x)}$, which is a polynomial.
\end{theorem}

\section{Applications}

In this section we will consider some examples of automatic power series of type
$$F(x)=\prod_{s=0}^{\infty}p(x^{l^{s}})=\sum_{i=1}^{\infty}c_ix^i,$$
where $p$ is a polynomial of degree $d$ with coefficients in $\mathbf{Q}$ and $l \geq 2$. It has been proved by Proposition 4.5 that the number of such polynomials $p$ is fixed once given the degree $d$ of the polynomial and $l$. But when $l$ and $d$ are both large, it will be difficult to compute the semi-group of matrix discussed in Section 4.2. Here we show a method applied on a particular example to generate the couples $(p,l)$ such that $\prod_{s=0}^{\infty}p(x^{l^{s}})=\sum_{i=1}^{\infty}c_ix^i$ is an automatic power sequence.

Let us consider firstly the power series $F_1(x)$ defined by $p_1(x)=1+x-x^3-x^4$ and $l=2$, it is easy to check that

$$F_1(x)=\prod_{s=0}^{\infty}p_1(x^{2^{s}})=\prod_{s=0}^{\infty}(1+x^{2^{s}})\prod_{s=0}^{\infty}(1-(x^3)^{2^{s}}).$$
And it is well known that $\prod_{s=0}^{\infty}(1+x^{2^{s}})=\frac{1}{1-x}=\sum_{i=1}^{\infty}x^i$ and $\prod_{s=0}^{\infty}(1-x^{2^{s}})=\sum_{i=1}^{\infty}b_nx^i$, where $(b_n)_{n \in \mathbf{N}}$ is the Thue-Morse sequence beginning with $1,-1$. So the coefficient of term $x^n$ in $F_1(x)$, say $f_1(n)$, can be calculated by 

$$f_1(n)=\sum_{3i\leq n}b_i.$$ 
The sequence $(f_1(n))_{n \in \mathbf{N}}$ is bounded because of the fact that $b_{2n+1}+b_{2n}=0$, so $F_1(x)$ is a $2$-automatic power sequence. Moreover the transition matrices $\Gamma_1$ and $\Gamma_0$ can be defined by

 $$\begin{pmatrix} 
c_{2n} \\
c_{2n-1}\\ 
c_{2n-2}\\
c_{2n-3} \\
\end{pmatrix}=
\Gamma_1\begin{pmatrix} 
c_{n} \\
c_{n-1}\\ 
c_{n-2}\\
c_{n-3} \\
\end{pmatrix}
=
\begin{pmatrix} 
1&0&-1&0\\
0&1&-1&0\\
0&1&0&-1\\
0&0&1&-1\\
\end{pmatrix}
\begin{pmatrix} 
c_{n} \\
c_{n-1}\\ 
c_{n-2}\\
c_{n-3} \\
\end{pmatrix}$$

 $$\begin{pmatrix} 
c_{2n+1} \\
c_{2n}\\ 
c_{2n-1}\\
c_{2n-2} \\
\end{pmatrix}=
\Gamma_0\begin{pmatrix} 
c_{n} \\
c_{n-1}\\ 
c_{n-2}\\
c_{n-3} \\
\end{pmatrix}=
\begin{pmatrix} 
1&-1&0&0\\
1&0&-1&0\\
0&1&-1&0\\
0&1&0&-1\\
\end{pmatrix}
\begin{pmatrix} 
c_{n} \\
c_{n-1}\\ 
c_{n-2}\\
c_{n-3} \\
\end{pmatrix}.$$

Remarking that 

$$\Gamma_0^2=\begin{pmatrix} 
1&-1&-1&1\\
0&0&-1&1\\
0&1&-2&1\\
0&1&-1&0\\
\end{pmatrix}, 
\Gamma_1\Gamma_0=\begin{pmatrix} 
1&-1&0&0\\
1&-1&-1&1\\
0&0&-1&1\\
0&1&-2&1\\
\end{pmatrix},$$
$$\Gamma_0\Gamma_1=\begin{pmatrix} 
1&-2&1&0\\
1&-1&0&0\\
1&-1&-1&1\\
0&0&-1&1\\

\end{pmatrix},
\Gamma_1^2=\begin{pmatrix} 
0&-1&1&0\\
1&-2&1&0\\
1&-1&0&0\\
1&-1&-1&1\\
\end{pmatrix},$$
let us consider the the power series $F_2(x)$ defined by $p_2(x)=1+x+x^2-x^4-x^5-2x^6-x^7-x^8+x^{10}+x^{11}+x^{12}=(x^2+x+1)(x^6-1)(x^4-1)$ and $l=4$, the transition matrices of this polynomial are

$$\alpha_0=\begin{pmatrix} 
1&-1&-1&1&0\\
0&0&-1&1&0\\
0&1&-2&1&0\\
0&1&-1&0&0\\
0&1&-1&-1&1\\
\end{pmatrix}
\alpha_1=\begin{pmatrix} 
1&-1&0&0&0\\
1&-1&-1&1&0\\
0&0&-1&1&0\\
0&1&-2&1&0\\
0&1&-1&0&0\\
\end{pmatrix}$$

$$\alpha_2=\begin{pmatrix} 
1&-2&1&0&0\\
1&-1&0&0&0\\
1&-1&-1&1&0\\
0&0&-1&1&0\\
0&1&-2&1&0\\
\end{pmatrix}
\alpha_3=\begin{pmatrix} 
0&-1&1&0&0\\
1&-2&1&0&0\\
1&-1&0&0&0\\
1&-1&-1&1&0\\
0&0&-1&1&0\\
\end{pmatrix}$$
If we define a sequence of matrices $(\alpha_n)_{n \in \mathbf{N}}$ by $\alpha_{4n+i}=\alpha_n\alpha_i, 0\leq i \leq 3$, then the $n$-th coefficient of $F_2(x)$ is $f_2(n)=\alpha_n(1,1)$. However the matrices $\alpha_i$ for $i=0,1,2,3$ are all of form $\begin{pmatrix} 
A_i&0\\
B_i&C_i\\
\end{pmatrix}$ with $A_i$ of size $4 \times 4$, $B_i$ of size $4 \times 1$, $C_i$ of size $1 \times 1$ and $0$ the $0$-matrix of size $1 \times 4$, so $\alpha_n(1,1)$ can be calculated only by the multiplications between $A_i$. Remarking that this four matrices are nothing else then $\Gamma_0^2, \Gamma_1\Gamma_0, \Gamma_0\Gamma_1, \Gamma_0^2$, we conclude that the sequence $(f_2(n))_{n \in \mathbf{N}}$ is bounded so $4$-automatic.

By the same method, the power series $F_3(x)$ defined by $p_3(x)=1+x+x^2-x^4-x^5+x^7+x^8-x^{10}-x^{11}-x^{12}=(x^2+x+1)(x^6+1)(1-x^4)$ and $l=4$ is also $4$-automatic. In fact, its transition matrices are

$$\beta_0=\begin{pmatrix} 
1&-1&1&-1&0\\
0&0&1&-1&0\\
0&1&0&-1&0\\
0&1&-1&0&0\\
0&1&-1&1&-1\\
\end{pmatrix}
\beta_1=\begin{pmatrix} 
1&-1&0&0&0\\
1&-1&1&-1&0\\
0&0&1&-1&0\\
0&1&0&-1&0\\
0&1&-1&0&0\\
\end{pmatrix}$$

$$\beta_2=\begin{pmatrix} 
1&0&-1&0&0\\
1&-1&0&0&0\\
1&-1&1&-1&0\\
0&0&1&-1&0\\
0&1&0&-1&0\\
\end{pmatrix}
\beta_3=\begin{pmatrix} 
0&1&-1&0&0\\
1&0&-1&0&0\\
1&-1&0&0&0\\
1&-1&1&-1&0\\
0&0&1&-1&0\\
\end{pmatrix}$$
and once more they are of form $\begin{pmatrix} 
A_i&0\\
B_i&C_i\\
\end{pmatrix}$ with $A_0=-\Gamma_0\Gamma_0\Gamma_1\Gamma_1,A_1=-\Gamma_1\Gamma_0\Gamma_0\Gamma_0,A_2=\Gamma_0\Gamma_1\Gamma_0\Gamma_0, A_3=\Gamma_1\Gamma_1\Gamma_0\Gamma_0$.

Furthermore, as $$\prod_{s=0}^{\infty}((x^2)^{4^s}+1)(x^{4^s}+1)=\prod_{s=0}^{\infty}\frac{(x^4)^{4^s}-1}{x^{4^s}-1}=\frac{1}{1-x}$$

we have $$(1-x)F_2(x)=\prod_{s=0}^{\infty}((x^2)^{4^s}+x^{4^s}+1)((x^6)^{4^s}-1)\frac{(x^4)^{4^s}-1}{((x^2)^{4^s}+1)(x^{4^s}+1)}=\prod_{s=0}^{\infty}(x^9)^{4^s}-(x^6)^{4^s}-(x^3)^{4^s}+1,$$
$$(1-x)F_3(x)=\prod_{s=0}^{\infty}((x^2)^{4^s}+x^{4^s}+1)((x^6)^{4^s}+1)\frac{(x^4)^{4^s}-1}{((x^2)^{4^s}+1)(x^{4^s}+1)}=\prod_{s=0}^{\infty}-(x^9)^{4^s}+(x^6)^{4^s}-(x^3)^{4^s}+1.$$

\begin{proposition}
The power series $$F_2(x)=\prod_{s=0}^{\infty}((x^2)^{4^s}+x^{4^s}+1)((x^6)^{4^s}-1)((x^4)^{4^s}-1)$$ and  $$F_3(x)=\prod_{s=0}^{\infty}((x^2)^{4^s}+x^{4^s}+1)((x^6)^{4^s}+1)(-(x^4)^{4^s}+1)$$ are $4$-automatic.
\end{proposition}

\bibliographystyle{alpha}
\bibliography{citations}
\end {document}